\documentclass[a4paper]{amsart}

\usepackage{tikz}
\usepackage{amsmath}
\usepackage{amssymb}
\usepackage{amsthm}
\usepackage{latexsym}
\usepackage{enumerate}
\usepackage{prettyref}
\usepackage{hyperref}

\newcommand{\Ao}{\ensuremath{\mathbb{A}^1}}

\newcommand{\Spec}{\ensuremath{\operatorname{Spec}}}

\newcommand{\smk}{\ensuremath{\operatorname{Sm}_k}}

\newcommand{\set}{\ensuremath{\mathcal{S}et}}

\newcommand{\hofib}{\ensuremath{\operatorname{hofib}}}
\newcommand{\Nis}{\ensuremath{\operatorname{Nis}}}
\newcommand{\Nil}{\ensuremath{\operatorname{Nil}}}
\newcommand{\tor}{\ensuremath{\operatorname{tor}}}

\newcommand{\classify}[1]{\ensuremath{B\haut({#1})}}

\newcommand{\haut}{\ensuremath{\operatorname{hAut}_\bullet}}

\newtheorem{mainthm}{Theorem}
\newtheorem{theorem}{Theorem}[section]
\newrefformat{thm}{\hyperref[{#1}]{Theorem~\ref*{#1}}}
\newtheorem{definition}[theorem]{Definition}
\newrefformat{def}{\hyperref[{#1}]{Definition~\ref*{#1}}}
\newtheorem{lemma}[theorem]{Lemma}
\newrefformat{lem}{\hyperref[{#1}]{Lemma~\ref*{#1}}}
\newtheorem{proposition}[theorem]{Proposition}
\newrefformat{prop}{\hyperref[{#1}]{Proposition~\ref*{#1}}}
\newtheorem{corollary}[theorem]{Corollary}
\newrefformat{cor}{\hyperref[{#1}]{Corollary~\ref*{#1}}}
\newtheorem{remark}[theorem]{Remark}
\newrefformat{rem}{\hyperref[{#1}]{Remark~\ref*{#1}}}
{
\newtheorem{examplecore}[theorem]{Example}}
\newrefformat{ex}{\hyperref[{#1}]{Example~\ref*{#1}}}

\newenvironment{example}{\begin{examplecore}}{\hspace*{\fill}
$\square$\par\vspace{.1cm}\end{examplecore}}

\begin{document}

\title{Units in Grothendieck-Witt rings and $\Ao$-spherical fibrations} 
\date{February 2013}

\author{Matthias Wendt}

\address{Matthias Wendt, Mathematisches Institut,
Albert-Ludwigs-Uni\-ver\-si\-t\"at Freiburg, Eckerstra\ss{}e 1, 79104, 
  Freiburg im Breisgau, Germany}
\email{matthias.wendt@math.uni-freiburg.de}


\begin{abstract}
In this note, we show that the units in Grothendieck-Witt rings extend
to an unramified strictly $\Ao$-invariant sheaf of abelian groups on
the category of smooth schemes. This implies that there is an
$\Ao$-local classifying space of spherical fibrations.  
\end{abstract}

\maketitle
\setcounter{tocdepth}{1}
\tableofcontents

\section{Introduction}

In this note, we discuss units in Grothendieck-Witt rings and
their relation to the theory of spherical fibrations. Via Morel's
generalization of the Brouwer degree \cite[Corollary
1.24]{morel:book}, the Grothendieck-Witt ring is the ring of
continuous self-maps of an algebraic sphere up to $\Ao$-homotopy. A
consequence of this is an identification of the group of homotopy
self-equivalences of a sphere with the units in the Grothendieck-Witt
ring.  
On the other hand, the general theory of localization of fibre
sequences of simplicial sheaves developed in \cite{flocal} implies
the existence of an $\Ao$-local classifying space of fibrations with
fibre $X$ provided the sheaf of homotopy self-equivalences is
strongly $\Ao$-local. These two results are the starting point for the
present work, in which we investigate strong $\Ao$-invariance of the
sheaf of units in Grothendieck-Witt rings. The main result is the
following: 

\begin{mainthm}
Let $k$ be an infinite perfect field of characteristic $\neq
2$. Associating to an extension field $L/k$ the abelian group
$GW(L)^\times$ of units in the Grothendieck-Witt ring extends to a
strictly $\Ao$-invariant sheaf of abelian groups on $\smk$. 
\end{mainthm}

There are several ingredients to the proof. First of all, the
structure of the group of units in Grothendieck-Witt rings is
well-known - it is a combination of square classes and units arising
from torsion in the fundamental ideal. Moreover, for the additive
groups of Grothendieck-Witt rings as well as the powers of the
fundamental ideal, strict $\Ao$-invariance and the corresponding
Gersten resolutions are already established \cite{bgpw}
resp. \cite[Chapters 4,5]{morel:book}. We 
describe a Gersten-type resolution of the units in the
Grothendieck-Witt rings, based on the explicit computation of
contractions.  

It is noteworthy and strange that this Gersten resolution has
underlying sets and set maps almost the Gersten resolution of the
torsion in the fundamental ideal, but the abelian group structure of
its terms is different from the usual addition in the Witt rings. 

From the strict $\Ao$-invariance of the units in Grothendieck-Witt
rings, we can establish the existence of an $\Ao$-local classifying
space of spherical fibrations.

\begin{corollary}
Let $S^{2n,n}$ denote an $\Ao$-local fibrant model of the motivic
sphere, and denote by $\classify{S^{2n,n}}$ the corresponding
classifying space of Nisnevich locally trivial torsors, as in
\cite{classify}. Then this space is in fact $\Ao$-local and classifies
Nisnevich locally trivial spherical fibrations.
\end{corollary}

The relation between strong $\Ao$-invariance of the units and the
existence of a classifying space of spherical fibrations was already
pointed out in \cite[Remark 8.3]{flocal}.
With this classifying space, we can discuss orientability and
orientation characters for 
spherical fibrations. There is also an obvious analogue of the
unstable $J$-homomorphism $BGL_n\rightarrow \classify{S^{2n,n}}$, whose
homotopy fibre controls reduction of spherical fibrations to vector
bundles. It turns out that non-trivial torsion in the Witt ring
implies the existence of spherical fibrations which can not be reduced
to vector bundles. This is a phenomenon which is not visible in either
real or complex realization. In particular, we find exotic
Poincar{\'e} dualities on the projective line $\mathbb{P}^1$ over
$\mathbb{Q}$, which are not induced from vector bundles.

\emph{Structure of the paper:}
We first recall some basic facts about the structure of
Grothendieck-Witt rings in \prettyref{sec:gw}. In \prettyref{sec:unr},
we show that the units in the Grothendieck-Witt ring can be extended
to an unramified sheaf of abelian groups. The main work is the
computation of contractions in \prettyref{sec:contract} and the
description of the Gersten resolution of the units in
\prettyref{sec:gersten}. Finally, in \prettyref{sec:sphere}, we
discuss some consequences for orientation theory and spherical
fibrations. 

\emph{Acknowledgements:} I would like to thank Aravind Asok and
Annette Huber for discussions about the results of the paper.

\section{Recollection on Grothendieck-Witt rings}
\label{sec:gw}

We recall some basic facts about the structure of Grothendieck-Witt
rings. The reference used will be \cite{knebusch:kolster}. The reader
should be aware that there are more modern and much more high-tech
approaches to Grothendieck-Witt groups available,
cf. \cite{schlichting}. In the present work, we will not need more
than the structure of units and the existence of residue maps.

Throughout the paper we assume that the base field $k$ has
characteristic $\neq 2$. 

\subsection{Basic definitions}

A bilinear space over $k$ is a pair $(V,\phi)$ of a $k$-vector space
$V$ and a symmetric bilinear form $\phi:V\times V\rightarrow k$. For
example, the hyperbolic plane $\mathbb{H}$ is the bilinear space
$(k^2,h)$ with  the hyperbolic bilinear form
$h((x_1,x_2),(y_1,y_2))=x_1y_2+x_2y_1$. 

Denote by $S(k)$ the set of isomorphism classes of bilinear
spaces over $k$. Orthogonal sum of bilinear spaces defines an addition
on $S(k)$ with respect to which $S(k)$ is a commutative monoid. There
is also a multiplication given by tensor product of bilinear spaces:
$$
(V_1,\phi_1)\otimes(V_2,\phi_2)=\left(V_1\otimes
V_2,\phi_1\otimes\phi_2:(e_1\otimes e_2,f_1\otimes f_2)\mapsto
\phi_1(e_1,f_1)\phi_2(e_2,f_2)\right). 
$$

Two bilinear spaces $V_1$ and $V_2$ are called \emph{stably
  isomorphic} if there exists a bilinear space $W$ such that 
$V_1\oplus W\cong V_2\oplus W$. 

\begin{definition}
The Grothendieck-Witt ring $GW(k)$ of the field $k$ is the group
completion of the commutative semiring of stable isomorphism classes
of bilinear spaces over $k$. The Witt ring $W(k)$ is the quotient of
the Grothendieck-Witt ring by the ideal $\mathbb{Z}\cdot[\mathbb{H}]$. 
\end{definition}

It can be shown, cf. \cite{knebusch:kolster}, that the
Grothendieck-Witt ring is the quotient of the group ring
$\mathbb{Z}[k^\times/(k^\times)^2]$ of the group of square classes
modulo the relations  
$$
((1)+(a))((1)-(1+a)), \qquad a\in k^\times\setminus\{-1\}.
$$
with $(a)$ denoting the square class of $a$ in the group ring. This
description via the group ring of square classes is the basis for
various structural theorems on Grothendieck-Witt rings,
cf. \cite{knebusch:kolster}. 

Finally, note that in characteristic $\neq 2$, the Grothendieck-Witt
ring can also be defined as the ring of stable isomorphism classes of
quadratic spaces, i.e. vector spaces equipped with a quadratic form. 

\subsection{Description of units}
Next, we recall from \cite{knebusch:kolster} the description of the
units of the Witt ring of a field:  

\begin{proposition}
\label{prop:units}
For every field $F$ of characteristic $\neq 2$, there is a pushout
square of abelian groups: 
\begin{center}
\begin{minipage}[c]{10cm}
\begin{tikzpicture}[scale=1.2,arrows=->]
\node (A01) at (0,1) {$F^\times/(F^\times)^2\cap (1+I(F)_{\tor})$};
\node (A) at (0,0) {$F^\times/(F^\times)^2$};
\node (GmGm) at (4,1) {$(1+I(F)_{\tor})$};
\node (X) at (4,0) {$W(F)^\times$};
\draw (A01) to  (A);
\draw (A) to (X);
\draw (A01) to (GmGm);
\draw (GmGm) to (X);
\end{tikzpicture}
\end{minipage}
\end{center}
\end{proposition}

\begin{proof}
The bottom morphism $F^\times/(F^\times)^2\rightarrow W(F)^\times$ is
the obvious one, mapping a square class $a\in F^\times/(F^\times)^2$
to the class of the corresponding one-dimensional symmetric bilinear
space in $W(F)^\times$.   

The torsion of the fundamental ideal is a subgroup
$I(F)_{\tor}\subseteq W(F)$. As a consequence of \cite[Propositions
2.12 and 2.15]{knebusch:kolster}, $I(F)_{\tor}=\Nil(W(F))$. For
$x,y\in \Nil(W(F))=I(F)_{\tor}$, we have $(1+x)(1+y)=1+x+y+xy$ with
$x+y+xy\in \Nil(W(F))=I(F)_{\tor}$. Moreover, for $x\in
\Nil(W(F))$ a nilpotent element, $(1+x)$ is invertible - this provides
the subgroup $(1+I(F)_{\tor})\hookrightarrow W(F)^\times$. 

By \cite[Corollary 2.25]{knebusch:kolster} every unit of $W(F)$ is of
the form $\pm\langle u\rangle (1+x)$ with $u\in F^\times$  and $x\in
\Nil(W(F))$ nilpotent. As in \cite[Remark 2.26]{knebusch:kolster}, the
square classes in $1+\Nil(W(F))$ are precisely the square classes
$\langle u\rangle$ with $u$ a sum of squares. 

The assertion then follows from these statements.
\end{proof}

We are actually interested in the units of the Grothendieck-Witt
ring. We use the above quotient presentation $GW(F)\rightarrow
W(F)$.  

\begin{corollary}
\label{cor:gww}
There is an exact sequence
$$
1\rightarrow \{\pm 1\}\rightarrow GW(F)^\times\rightarrow
W(F)^\times\rightarrow 1.
$$
\end{corollary}

\begin{proof}
This follows from \cite[Proposition 2.24]{knebusch:kolster} together
with the fact that $-1\in GW(F)$ is not a square class. It becomes a
square class in $W(F)$ since $-1=\langle-1\rangle$.
\end{proof}

\begin{remark}
Note that in general not all units of $W(F)$ are represented by square
classes. For example, if $p$ is an odd prime, the cardinality of
$W(\mathbb{Q}_p)$ is $16$, with $W(\mathbb{Q}_p)^\times$ of size $8$
but only $4$ square classes. 
In this case, we have
$$
W(\mathbb{Q}_p)^\times/(\mathbb{Q}_p^\times/(\mathbb{Q}_p^\times)^2)\cong
(1+I(\mathbb{Q}_p)_{\tor})/(\mathbb{Q}_p^\times/(\mathbb{Q}_p^\times)^2\cap
(1+I(\mathbb{Q}_p)_{\tor}))\cong\mathbb{Z}/2.
$$
We describe a representative. 

Let $a,b\in\mathbb{Q}_p$ with $(a,b)_p=-1$ and consider the norm form
of the quaternion algebra for $(a,b)$:
$$
f(X,Y,Z,T)=Z^2-aX^2-bY^2+abT^2.
$$
This is an element of $I(\mathbb{Q}_p)_{\tor}$ and has invariants
$d(f)=1$ and $\epsilon(f)=-(-1,-1)$. 
On the other hand, any $5$-dimensional form represented by a square
class is of the form $g\sim[a]\oplus\mathbb{H}^2$, and therefore has
invariants $d(g)=a$ and $\epsilon(g)=(-1,-1)$. We see that $1\oplus f$ -
which has the same invariants as $f$ - is a unit of $W(\mathbb{Q}_p)$
not represented by a square class. 

Of course, similar examples can be constructed over $\mathbb{Q}$ using
the Hasse-Minkowski local-global principle. This way, we see that 
$W(\mathbb{Q})^\times/(\mathbb{Q}^\times/(\mathbb{Q}^\times)^2)$
is infinite.

Conversely, note that $I(\mathbb{R})_{\tor}=0$, therefore of the units
in $W(\mathbb{R})=\mathbb{Z}$, one is of the form $1+x$ with $x$
nilpotent, but $-1$ is not.
\end{remark}

\section{The units as unramified sheaf}
\label{sec:unr}

In this section, we will show that the units in Grothendieck-Witt
rings extend to an unramified sheaf of abelian groups on the category
$\smk$ of smooth schemes over the field $k$, in the sense of
\cite[Definition 2.1]{morel:book}. For that, we use the correspondence
between unramified sheaves and unramified $\mathcal{F}_k$-data from
\cite[Section 2]{morel:book}. The data used for the units
(i.e. residue and specialization morphisms) are the same as those for
the Grothendieck-Witt rings, which are in fact compatible with the
ring structure. 

In this section, we assume the base field to be infinite perfect, and
we denote by $\mathcal{F}_k$ the category of field extensions $L/k$
such that $L$ has finite transcendence degree over $k$.  For a
discrete valuation $v$ on the field $L$, we denote by $\kappa(v)$ the 
corresponding residue field.

\subsection{Recollection on unramified sheaves}

We just recall the main points of \cite[Section
2.1]{morel:book}. 
The central definitions in \cite[Section 2.1]{morel:book} are the
definition of \emph{unramified presheaves of sets} in Definition 2.1,
and the definition of \emph{unramified $\mathcal{F}_k$-data} in the combined
Definition 2.6 and Definition 2.9.
The main result of the section is Theorem 2.11 which states that there
is an equivalence of categories between the unramified
$\mathcal{F}_k$-data and unramified sheaves of sets on $\smk$.
The functor from unramified sheaves on $\smk$ to $\mathcal{F}_k$-data
is given by evaluation - all finitely generated field extensions of
$k$ are function fields of smooth schemes over $k$. The functor in the
other direction, from $\mathcal{F}_k$-data to unramified sheaves on
$\smk$ is given by taking the unramified elements: for irreducible $X$
we take 
$$
X\mapsto \mathcal{S}(X)=\bigcap_{x\in
  X^{(1)}}\mathcal{S}(\mathcal{O}_x)\subseteq \mathcal{S}(k(X))
$$
and for general $X\in\smk$, we set
$\mathcal{S}(X)=\prod_{i\in X^{(0)}}\mathcal{S}(X_i)$.

\subsection{Residue and specialization morphisms}

In this section, we recall the residue morphisms on Grothendieck-Witt
rings. We  will mostly follow \cite{milnor} and \cite{morel:book}. The
residue morphisms provide us with the data (D1)-(D3) for unramified
$\mathcal{F}_k$-data. The first such datum is obvious - (D1) just
requires that we have induced morphisms for field extensions. The
obvious functorial homomorphisms for Grothendieck-Witt rings are ring
homomorphisms, which provides data (D1)  for the units of
Grothendieck-Witt rings. We note that $GW^\times$ is a continuous
functor: the Grothendieck-Witt rings are continuous functors,
therefore also $GW^\times$ commutes with filtered colimits. 

For the datum (D2), we need to specify the unramified elements. 
Recall from \cite[Theorem 3.15]{morel:book} that there are
residue homomorphisms $\partial^\pi_v:K^{MW}_\bullet(F)\rightarrow
K^{MW}_{\bullet-1}(\kappa(v))$ for any discrete valuation $v$ on $F$
with valuation ring $\mathcal{O}_v$, residue field $\kappa(v)$ and
choice of uniformizer $\pi$. By \cite[Lemma 3.19]{morel:book}, the
kernel of such a residue homomorphism does not depend on the choice of
uniformizer $\pi$, and is denoted by
$\underline{\mathbf{K}}^{MW}_n(\mathcal{O}_v)$. Finally, \cite[Theorem
3.22]{morel:book} states that this
$\underline{\mathbf{K}}^{MW}_\bullet(\mathcal{O}_v)$ is the subring
generated by $\eta$ and the symbols $[u]\in K^{MW}_1(F)$ with $u\in
\mathcal{O}_v^\times$. Therefore, we can set 
$$
\mathbf{GW}^\times(\mathcal{O}_v)=
\underline{\mathbf{K}}^{MW}_0(\mathcal{O}_v)^\times.
$$
In particular, the unramified elements form an abelian subgroup of
$GW(F)^\times$.  

Finally, we need the datum (D3), i.e. for any $F\in\mathcal{F}_k$ and
any discrete valuation $v$ on $F$ a specialization morphism
$s_v:\mathbf{GW^\times}(\mathcal{O}_v)\rightarrow\mathbf{GW^\times}(\kappa(v))$. 
For this, we have to unwind the proof of \cite[Lemma
2.36]{morel:book}, because the fact that Milnor-Witt K-theory
sheaves are unramified is deduced from other axioms which only make
sense in the $\mathbb{Z}$-graded situation which we do not have here.
The datum (D3) on $\underline{\mathbf{K}}^{MW}_0$ is given by 
$$
s_v:\underline{\mathbf{K}}^{MW}_0(\mathcal{O}_v)\rightarrow
K^{MW}_0(\kappa(v)):\alpha\mapsto\partial^\pi_v([\pi]\alpha)
$$ 
and is in fact independent of the choice of a $\pi$. We compute the value
of $s^\pi_v$ on generators $1+\eta[u]$ with
$u\in\mathcal{O}_v^\times$. From \cite[Theorem 3.15]{morel:book}, we find 
$$
s^\pi_v(\langle u\rangle)=\partial^\pi_v([\pi](1+\eta[u]))=
1+\eta[\overline{u}]=\langle \overline{u}\rangle.
$$
From \cite[Lemma 3.16]{morel:book}, this is in fact a morphism of
rings on $\underline{\mathbf{K}}^{MW}_0(\mathcal{O}_v)$. In
particular, the reduction morphism $\mathcal{O}_v^\times\rightarrow
\kappa(v)^\times$ induces a morphism of units 
$$
s_v:\underline{\mathbf{K}}^{MW}_0(\mathcal{O}_v)^\times\rightarrow
K^{MW}_0(\kappa(v))^\times.
$$ 
This is the relevant datum (D3). 

\begin{remark}
As in \cite[Theorem 3.22]{morel:book},
$\underline{\mathbf{K}}^{MW}_0(\mathcal{O}_v)$ is the subring of
$K^{MW}_0(F)$ generated by the symbols $\langle u\rangle=1+\eta[u]$
with $u\in \mathcal{O}_v^\times$. 
In particular, the morphism
$$
\mathbb{Z}[\mathcal{O}_v^\times/(\mathcal{O}_v^\times)^2]\rightarrow 
K^{MW}_0(F):u\mapsto \langle u\rangle
$$
factors through $\underline{\mathbf{K}}^{MW}_0(\mathcal{O}_v)$. It is
not clear to me if the relations usually imposed in the definition of
the Grothendieck-Witt ring generate the kernel of the above morphism.
\end{remark}

\begin{proposition}
\label{prop:intersect}
Let $F\in\mathcal{F}_k$, and let $v$ be a discrete valuation on $F$. 
We have 
$$
GW(F)^\times\cap
\mathbf{GW}(\mathcal{O}_v)=\mathbf{GW^\times}(\mathcal{O}_v).
$$
\end{proposition}

\begin{proof}
Only the $\subseteq$-direction needs a proof. We use the
square of \prettyref{prop:units}.  
For units coming from $F^\times/(F^\times)^2$, the inclusion is
true: if such a unit is unramified, then it is the square class of a
unit in $\mathcal{O}_v$, hence invertible in
$\mathbf{GW}(\mathcal{O}_v)$.   

Now let $1+a\in GW(F)^\times$, i.e. $a\in I(F)_{\tor}$. Assume
$1+a\in\mathbf{GW}(\mathcal{O}_v)$, hence
$a\in\mathbf{I_{\tor}}(\mathcal{O}_v)$. Its inverse in $GW(F)^\times$ is 
given by $1+\sum_{i=1}^{n-1}(-a)^i$, where $n$ is the smallest positive
number such that $a^n=0$. But then
$\sum_{i=1}^{n-1}(-a)^i\in\mathbf{I_{\tor}}(\mathcal{O}_v)$, hence $1+a\in
\mathbf{GW^\times}(\mathcal{O}_v)$. 
\end{proof}

\subsection{Verification of the axioms}

Our next goal is the verification of the axioms (A1)-(A4) for
unramified $\mathcal{F}_k$-data from \cite[Chapter 2]{morel:book}.

\begin{lemma}[Axiom (A1)]
Let $\iota:E\hookrightarrow F$ be a separable extension in
$\mathcal{F}_k$, and let $v$ be a discrete valuation on $F$ which
restricts to a discrete valuation $w$ on $E$ with ramification index
$1$. Then $GW(\iota)^\times$ maps $\mathbf{GW^\times}(\mathcal{O}_w)$ into
$\mathbf{GW^\times}(\mathcal{O}_v)$. If furthermore $\iota$ induces an
isomorphism $\overline{\iota}:\kappa(w)\rightarrow \kappa(v)$ of residue
fields, then the following square is cartesian:  
\begin{center}
\begin{minipage}[c]{10cm}
\begin{tikzpicture}[scale=1.2,arrows=->]
\node (A01) at (0,1) {$\mathbf{GW^\times}(\mathcal{O}_w)$};
\node (A) at (0,0) {$GW(E)^\times$};
\node (GmGm) at (3,1) {$\mathbf{GW^\times}(\mathcal{O}_v)$};
\node (X) at (3,0) {$GW(F)^\times$};
\draw (A01) to  (A);
\draw (A) to (X);
\draw (A01) to (GmGm);
\draw (GmGm) to (X);
\end{tikzpicture}
\end{minipage}
\end{center}
\end{lemma}

\begin{proof}
As $\mathbf{GW}(\mathcal{O}_w)=
\underline{\mathbf{K}}^{MW}_0(\mathcal{O}_w)$ is the subring
generated by symbols $\langle u\rangle$ with $u\in \mathcal{O}_w$ and
the ring homomorphism $GW(E)\rightarrow GW(F)$ is induced by $\langle
x\rangle\mapsto \langle \iota(x)\rangle$ for $x\in E^\times$, we find that
$GW(E)^\times\rightarrow GW(F)^\times$ restricts to a homomorphism
$\mathbf{GW}^\times(\mathcal{O}_w)\rightarrow
\mathbf{GW}^\times(\mathcal{O}_v)$.   

The second assertion states that $\mathbf{GW^\times}(\mathcal{O}_w)$
is the preimage of $\mathbf{GW^\times}(\mathcal{O}_v)$ under the
homomorphism $GW(E)^\times\rightarrow GW(F)^\times$. Axiom (A1) for
$GW$ states that $\mathbf{GW}(\mathcal{O}_w)$ is the
preimage of $\mathbf{GW}(\mathcal{O}_v)$ under the ring homomorphism
$GW(E)\rightarrow GW(F)$. Then we apply
\prettyref{prop:intersect}. Alternatively, given an element $x\in
GW(E)^\times$ which under the homomorphism lands in
$\mathbf{GW^\times}(\mathcal{O}_v)$. This means that both $x$ and
$x^{-1}$ are unramified in $GW(F)$. By Axiom (A1) for $GW$, both $x$
and $x^{-1}$ lie in $\mathbf{GW}(\mathcal{O}_w)$, hence $x\in
\mathbf{GW^\times}(\mathcal{O}_w)$. 
\end{proof}

\begin{lemma}[Axiom (A2)]
For an irreducible smooth scheme $X$ over $k$ with function field $F$,
every element $x\in GW(F)^\times$ lies in all but a finite number of
$\mathbf{GW^\times}(\mathcal{O}_v)$, with $v\in X^{(1)}$ running
through the (valuations for) codimension $1$ points of $X$. 
\end{lemma}

\begin{proof}
Any element in $K^{MW}_0(F)$ lies in all but a
finite number of $\underline{\mathbf{K}}^{MW}_0(\mathcal{O}_v)$. 
For $x\in GW(F)^\times$, both $x$ and $x^{-1}$ lie in all but a finite
number of $\mathbf{GW^\times}(\mathcal{O}_v)$, proving the claim. 
\end{proof}

\begin{lemma}[Axiom (A3)]
Let $\iota:E\hookrightarrow F$ be a separable extension in
$\mathcal{F}_k$, and let $v$ be a discrete valuation on $F$.
\begin{enumerate}[(i)]
\item
If $v$ restricts to a discrete valuation $w$ on $E$ with ramification
index $1$, then the following diagram is commutative: 
\begin{center}
\begin{minipage}[c]{10cm}
\begin{tikzpicture}[scale=1.2,arrows=->]
\node (A01) at (0,1) {$\mathbf{GW^\times}(\mathcal{O}_w)$};
\node (A) at (0,0) {$GW(\kappa(w))^\times$};
\node (GmGm) at (3,1) {$\mathbf{GW^\times}(\mathcal{O}_v)$};
\node (X) at (3,0) {$GW(\kappa(v))^\times$};
\draw (A01) to node [midway,left] {$s_w$} (A);
\draw (A) to (X);
\draw (A01) to (GmGm);
\draw (GmGm) to node [midway,right] {$s_v$} (X);
\end{tikzpicture}
\end{minipage}
\end{center}
\item If $v$ restricts to $0$ on $E$, then the image of
  $GW(E)^\times$ is contained in $\mathbf{GW^\times}(\mathcal{O}_v)$
  and - denoting $j:E\hookrightarrow \kappa(v)$ - we have a
  commutative triangle: 
\begin{center}
\begin{minipage}[c]{10cm}
\begin{tikzpicture}[scale=1.2,arrows=->]
\node (A) at (0,1) {$GW(E)^\times$};
\node (B) at (4,1) {$\mathbf{GW^\times}(\mathcal{O}_v)$};
\node (C) at (3,0) {$GW(\kappa(v))^\times$};
\draw (A) to (B);
\draw (B) to node [midway, right] {$s_v$} (C);
\draw (A) to node [midway, below] {$GW(j)$} (C);
\end{tikzpicture}
\end{minipage}
\end{center}
\end{enumerate}
\end{lemma}

\begin{proof}
Part (i) again can be verified using the whole Grothendieck-Witt
ring. The ring $\mathbf{GW}(\mathcal{O}_w)$ is generated by elements
of the form $\langle u\rangle$ with $u\in\mathcal{O}_w$. Commutativity
of the corresponding diagram for $GW$ instead of $GW^\times$ reduces
to the obvious $\overline{\iota}(\langle\overline{u}\rangle)=
\langle\overline{\iota(u)}\rangle$.  Since all morphisms in the
diagram for $GW$ respect the ring structure, we obtain a commutative
diagram of units, as required. 

For part (ii), we note again that the group $GW(E)$ is generated by
$\langle u\rangle$ with $u\in E^\times$. Since $v(u)=0$,
i.e. $\iota(u)\in\mathcal{O}_v^\times$, we have
$\langle\iota(u)\rangle\in
\underline{\mathbf{K}}^{MW}_0(\mathcal{O}_v)$. In particular
the image $GW(E)^\times$ in $GW(F)^\times$ is contained in
$\underline{\mathbf{K}}^{MW}_0(\mathcal{O}_v)^\times$. The
homomorphism denoted by $s_v\circ \mathcal{S}(i)$ in \cite{morel:book}
is therefore $\langle u\rangle\mapsto
\langle\overline{\iota(u)}\rangle$. The homomorphism denoted by
$j:E\rightarrow \kappa(v)$ equals $u\mapsto
\overline{\iota(u)}$. Again, everything is compatible with the ring
structures, hence we have verified (A3ii).
\end{proof}

\begin{lemma}[Axiom (A4)]
\begin{enumerate}[(i)]
\item Let $X$ be an essentially smooth local scheme of dimension $2$,
  let $z$ be the closed point of $X$ and let $y_0$ be a codimension $1$
  point with essentially smooth closure. Then the specialization 
  $s_{y_0}:\mathbf{GW^\times}(\mathcal{O}_{y_0})\rightarrow
  GW(\kappa(y_0))^\times$ maps $\bigcap_{y\in
    X^{(1)}}\mathbf{GW^\times}(\mathcal{O}_y)$ into
    $\mathbf{GW^\times}(\mathcal{O}_{\overline{y_0},z})\subseteq
    GW(\kappa(y_0))^\times$. 
\item The composition
$$
\bigcap_{y\in X^{(1)}}\mathbf{GW^\times}(\mathcal{O}_v)\rightarrow
\mathbf{GW^\times}(\mathcal{O}_{\overline{y_0},z})\rightarrow
\mathbf{GW^\times}(\kappa(z))
$$
does not depend on the choice of a codimension $1$ point $y_0$ with
essentially smooth closure. 
\end{enumerate}
\end{lemma}

\begin{proof}
Again, we know Axiom (A4) for the unramified sheaf  $\mathbf{GW}$ of
Grothendieck-Witt rings. 

In (i), the specialization morphism $\mathbf{GW}(\mathcal{O}_{y_0})\rightarrow
GW(\kappa(y_0))$  is a ring homomorphism, cf. \cite[Lemma
3.16]{morel:book}, and the specialization morphisms for units 
$s_{y_0}:\mathbf{GW^\times}(\mathcal{O}_{y_0})\rightarrow
GW(\kappa(y_0))^\times$ is induced from the corresponding
specialization morphism for $GW$. By Axiom (A4) for $\mathbf{GW}$,
$\bigcap_{y\in X^{(1)}}\mathbf{GW}(\mathcal{O}_y)$ lands in
$\mathbf{GW}(\mathcal{O}_{\overline{y_0},z})\subseteq
GW(\kappa(y_0))$. Therefore, $\mathbf{GW^\times}(\mathcal{O}_{y_0})$
lands in $\mathbf{GW}(\mathcal{O}_{\overline{y_0},z})\cap
GW(\kappa(y_0))^\times$. The conclusion follows from
\prettyref{prop:intersect}. 

The composition in (ii) is a composition of specialization
morphisms. As mentioned before, these are the restrictions of the
specialization morphisms from $\mathbf{GW}$ to the group of
units. Axiom (A4ii) for $\mathbf{GW}$ states that the composition 
$$
\bigcap_{y\in X^{(1)}}\mathbf{GW}(\mathcal{O}_v)\rightarrow
\mathbf{GW}(\mathcal{O}_{\overline{y_0},z})\rightarrow
\mathbf{GW}(\kappa(z))
$$
is independent of the choice of $y_0$. Therefore, so is the
restriction of this morphism to the groups of units. 
\end{proof}

Having verified the axioms, the following is now a consequence of
\cite[Theorem 2.11]{morel:book}.

\begin{proposition}
Let $k$ be an infinite perfect field of characteristic $\neq 2$. The
assignment  
$$
GW^\times:\mathcal{F}_k\rightarrow\set:L\mapsto GW(L)^\times
$$
together with the data (D1)-(D3) at the beginning of the section
satisfy the axioms for an unramified $\mathcal{F}_k$-datum. In
particular, $GW^\times$ extends to an unramified sheaf of abelian
groups on $\smk$. 
\end{proposition}

\begin{remark}
Note that what we have proved above is: there is a reasonable notion
of unramified sheaf of rings in which all the data (D1)-(D3) are
compatible with the ring structures, and an unramified sheaf of rings
has an unramified sheaf of units. In lack of examples other than the
one above, we chose not to axiomatize this.
\end{remark}

\section{Contractions}
\label{sec:contract}

In this section, we want to compute the contractions of the unramified
sheaf $\mathbf{GW^\times}$. Recall from \cite{morel:book}, that for a
presheaf of groups $\mathcal{G}$ on $\smk$, one can define its
contraction  
$$
\mathcal{G}_{-1}:\smk^{\operatorname{op}}\rightarrow\mathcal{G}rp:X\mapsto
\ker\left(\operatorname{ev}_1:\mathcal{G}(\mathbb{G}_m\times
  X)\rightarrow\mathcal{G}(X)\right),
$$
where the morphism $\operatorname{ev}_1$ is induced from the inclusion
$1:\Spec(k)\hookrightarrow \mathbb{G}_m$.

In the course of computing contractions, we will need a modified
version of Witt groups. We introduce some notation:

\begin{definition}
\label{def:halfwitt}
Let $k$ be a field. We denote by $W(k)_{\tor}$ the torsion subgroup of 
the Witt ring. For $n\in\mathbb{N}$, we denote by $W(k)_{\tor}^{(n)}$
the set of torsion elements in $W(k)$ with the following modified
addition 
$$
a\boxplus_n b=a+b+(-2)^n\cdot a\cdot b,
$$
where $+$ and $\cdot$ on the right are the usual addition and
multiplication in the Witt ring.
\end{definition}

\begin{lemma}
The addition $\boxplus_n$ defined above turns $W(k)_{\tor}^{(n)}$ into an
abelian group. 
\end{lemma}

\begin{proof}
The product of torsion elements is again a torsion element, hence the
modified addition is in fact an operation on $W(k)_{\tor}$.
Associativity is 
\begin{eqnarray*}
(a\boxplus_n b)\boxplus_n c&=&
a+b+c+(-2)^n\left(ab+ac+bc+(-2)^nabc\right)\\
&=&a\boxplus_n(b\boxplus_n c).
\end{eqnarray*}
The neutral element is $0$, since $0\boxplus_n a=0+a+(-2)^n\cdot 0\cdot
a$. 
Commutativity is clear from the defining formula and commutativity of
the Witt ring multiplication.
Finally, the $\boxplus_n$-inverse of $a$ is given by 
$$
-\sum_{i=1}^j (-1)^{(i-1)(n-1)} 2^{n(i-1)}a^i.
$$
This is the usual formula for inverting a nilpotent element, the bound
$j$ for the summation is the smallest natural number such that
$(-1)^{(j-1)(n-1)} 2^{nj}a^{j+1}=0$. 
\end{proof}

In the following, we use the notation $H^1_v$ from \cite[Section
3]{morel:book}. 

\begin{lemma}
\label{lem:contract}
Let $F$ be a field in $\mathcal{F}_k$, let $v$ be a discrete
valuation on $F$ and let $\pi$ be a uniformizing element for $v$. Then
there exist isomorphisms
\begin{enumerate}[(i)]
\item
$H^1_v(\mathcal{O}_v,\mathbb{G}_m/2)=
F^\times/((F^\times)^2\cdot \mathcal{O}_v^\times)
\stackrel{\cong}{\longrightarrow} \mathbb{Z}/2$, and
\item
$H^1_v(\mathcal{O}_v,1+\mathbf{I_{\tor}})=
(1+I(F)_{\tor})/(1+\mathbf{I_{\tor}}(\mathcal{O}_v))
\stackrel{\cong}{\longrightarrow} W(\kappa(v))_{\tor}^{(1)}$.
\end{enumerate}
\end{lemma}

\begin{proof}
The isomorphism (i) is clear: every element of $x\in F^\times$ can be
written as $x=\pi^n\cdot u$ with $u\in\mathcal{O}_v^\times$, and the
isomorphism is given by mapping $x=\pi^nu\mapsto n\mod 2$.

For (ii), we need a little more work. Recall (e.g. from \cite[Proof of
Corollary 5.2]{milnor}) the definition of the residue homomorphism
$$
\partial: I(F)\stackrel{\partial}{\longrightarrow}
W(\kappa(v)):\langle \pi^n u\rangle\mapsto\left\{ 
\begin{array}{ll}
\langle\overline{u}\rangle & 2\nmid n\\
0 & \textrm{otherwise}
\end{array}\right.
$$
This is evidently an epimorphism, and its kernel is
$\mathbf{I}(\mathcal{O}_v)$. The residue morphism restricts to a map
$$
\partial:1+I(F)_{\tor}\rightarrow
W(\kappa(v))_{\tor}^{(1)}:1+a\mapsto \partial(a), 
$$
where the group structure on the source is multiplication in the Witt
ring, the group structure on the target is the modified addition of
\prettyref{def:halfwitt}. We need to establish that we have a
homomorphism.  
First note that $GW(F)$ is generated by $\langle x\rangle$ with $x\in
F^\times$. Any such $x$ can be written as $x=\pi^nu$ with $u\in
\mathcal{O}_v^\times$. Since  $\langle \pi^2\rangle=1$, any
element of $1+I(F)_{\tor}$ can be written as follows:
\begin{eqnarray*}
1+a+\langle\pi\rangle b&=&1+a+b+(\langle\pi\rangle-1)b\\
&\equiv & 1+(\langle\pi\rangle-1)b(1+a+b)^{-1}\mod
\mathbf{GW^\times}(\mathcal{O}_v).
\end{eqnarray*}
Note that we have used here that if $a+\langle\pi\rangle b$ is
torsion, then (via $\partial$) also $b$ and hence $a$ are torsion, so
$1+a+b$ is in fact invertible. The above decomposition is unique up to
multiplication with elements from
$\mathbf{GW^\times}(\mathcal{O}_v)$. Now let
$(1+(\langle\pi\rangle-1)a)$ and $(1+(\langle\pi\rangle-1)b)$ be 
such elements in $1+I(F)_{\tor}$. We find 
\begin{eqnarray*}
(1+(\langle\pi\rangle-1)a)\cdot(1+(\langle\pi\rangle-1)b)&=&
1+(\langle\pi\rangle-1)(a+b)+(\langle\pi\rangle-1)^2ab\\
&=&1+(\langle\pi\rangle-1)(a+b-2ab), 
\end{eqnarray*}
where we use $(\langle\pi\rangle-1)^2=2-2\langle\pi\rangle$. Applying
$\partial$ to the above equation, we have on the left side the
reduction of the product, and on the right side, we have
$\boxplus_1$-sum $a+b-2ab$. Hence, $\partial$ is a homomorphism.  

For injectivity, we note that the kernel of $\partial$  is identified
with  $1+\mathbf{I_{\tor}}(\mathcal{O}_v)$, since the set
$\mathbf{I_{\tor}}(\mathcal{O}_v)$ of unramified elements of
$I(F)_{\tor}$ is exactly the kernel of
$\partial:I(F)_{\tor}\rightarrow W(\kappa(v))_{\tor}$. 

Finally, we need surjectivity of $\partial$. Let  $\overline{a}\in  
W(\kappa(v))_{\tor}$ be an element. The ring $\mathcal{O}_v$ is smooth
and essentially of finite type over a field
$L\subseteq\mathcal{O}_v$. By \cite[Proposition 2]{lindel},
there exists a subring $A\hookrightarrow \mathcal{O}_v$ such that
$A=L[X]_{\mathfrak{m}}$ with $\mathfrak{m}\subseteq L[X]$ a maximal
ideal. Moreover, the inclusion $A\subseteq \mathcal{O}_v$ has the
following properties: in the corresponding field extension
$\operatorname{frac}(A)\hookrightarrow F$, the valuation $v$ on $F$
restricts to the valuation of $\operatorname{frac}(A)$ having $A$ as
valuation ring, the ramification index is $1$ and the inclusion
$A\hookrightarrow\mathcal{O}_v$ induces an isomorphism of residue
fields. Now we use Milnor's split exact sequence \cite[Theorem
5.3]{milnor} 
$$
0\rightarrow I(L)\rightarrow I(L(T))\rightarrow
\bigoplus_{\pi} W(L[T]/(\pi))\rightarrow 0. 
$$
Let $\pi\in L[T]$ be a generator of the maximal ideal $\mathfrak{m}$
above. The element $\overline{\alpha}$ can be viewed as an element of
$W(L[T]/(\pi))_{\tor}$, since the inclusion
$A\hookrightarrow\mathcal{O}_v$ induces an isomorphism of residue
fields. Because Milnor's sequence is split exact, there exists an
element $\alpha\in I(L(T))_{\tor}$ mapping to $\overline{\alpha}\in
W(L[T]/(\pi))_{\tor}$ under the corresponding residue
homomorphism. The image of $1+\alpha$ under the canonical homomorphism
$1+I(\operatorname{frac}(A))_{\tor}\rightarrow 1+I(K)_{\tor}$ provides a
$\partial$-preimage of $\overline{\alpha}$.
\end{proof}

Note that the first isomorphism in the statement is independent of
the choice of a $\pi$, the second is not. 

\begin{proposition}
\label{prop:contract}
Let $F$ be a field in $\mathcal{F}_k$, let $v$ be a discrete
valuation on $F$ and let $\pi$ be a uniformizing element for $v$. Then
there exists an isomorphism 
$$
H^1_v(\mathcal{O}_v,\mathbf{GW^\times})=
GW(F)^\times/\mathbf{GW^\times}(\mathcal{O}_v)
\stackrel{\cong}{\longrightarrow} 
\mathcal{A}\oplus W(\kappa(v))_{\tor}^{(1)},
$$
where $\mathcal{A}=0$ if $\pi$ is a sum of squares in
$F$, and $\mathcal{A}\cong\mathbb{Z}/2$ otherwise.
\end{proposition}

\begin{proof}
We first prove the result for $H^1_v(\mathcal{O}_v,\mathbf{W^\times})=
W(F)^\times/\mathbf{W^\times}(\mathcal{O}_v)$ instead. In that case,
the result will follow from the pushout square of
\prettyref{prop:units} and the previous lemma. First, we claim
that the intersection of the pushout square with
$\mathbf{W^\times}(\mathcal{O}_v)$ results in the corresponding pushout
square of the unramified subgroups:
\begin{center}
\begin{minipage}[c]{10cm}
\begin{tikzpicture}[scale=1.2,arrows=->]
\node (A01) at (0,1) {$\mathcal{O}_v^\times/(\mathcal{O}_v^\times)^2
\cap (1+\mathbf{I_{\tor}}(\mathcal{O}_v))$};
\node (A) at (0,0) {$\mathcal{O}_v^\times/(\mathcal{O}_v^\times)^2$};
\node (GmGm) at (4,1) {$(1+\mathbf{I_{\tor}}(\mathcal{O}_v))$};
\node (X) at (4,0) {$\mathbf{W^\times}(\mathcal{O}_v)$};
\draw (A01) to  (A);
\draw (A) to (X);
\draw (A01) to (GmGm);
\draw (GmGm) to (X);
\end{tikzpicture}
\end{minipage}
\end{center}
An inspection of the proof of \cite[Proposition
2.24]{knebusch:kolster} shows that the unramified units are precisely
those that can be written as product of an unramified square class and
an unramified element from $(1+\mathbf{I_{\tor}})$: if a unit is
unramified, we can write it as sum $\sum\langle u_i\rangle$ with
$u_i\in\mathcal{O}_v^\times$. In loc.cit., the element can be
decomposed multiplicatively as the product of the square class
$(-1)^n\prod\langle u_i\rangle$ and an element of the form $1+a$ with
$a$ nilpotent. In particular, both factors have to be unramified.
Therefore, intersecting the pushout square with
$\mathbf{W^\times}(\mathcal{O}_v)$ indeed produces exactly the
unramified subgroups. 

It follows, that the quotient of
the squares is also a pushout square:
\begin{center}
\begin{minipage}[c]{10cm}
\begin{tikzpicture}[scale=1.2,arrows=->]
\node (A01) at (0,1) {$\mathcal{S}$};
\node (A) at (0,0) {$H^1_v(\mathcal{O}_v,\mathbb{G}_m/2)\cong\mathbb{Z}/2$};
\node (GmGm) at (4,1)
{$H^1_v(\mathcal{O}_v,1+\mathbf{I_{\tor}})\cong W(\kappa(v))_{\tor}^{(1)}$}; 
\node (X) at (4,0) {$H^1_v(\mathcal{O}_v,\mathbf{W^\times})$};
\draw (A01) to  (A);
\draw (A) to (X);
\draw (A01) to (GmGm);
\draw (GmGm) to (X);
\end{tikzpicture}
\end{minipage}
\end{center}
The identifications in the lower left and upper right corner follow
from \prettyref{lem:contract}. Similarly, the $\mathcal{S}$ in the
upper left corner is $H^1_v(\mathcal{O}_v,\mathbb{G}_m/2\cap
(1+\mathbf{I_{\tor}}))$. As can be seen from the first isomorphism of
\prettyref{lem:contract}, this is $\mathbb{Z}/2$ or $0$ if
$\pi\in(1+I(F)_{\tor})$ or not, respectively. But the intersection 
$F^\times/(F^\times)^2\cap (1+I(F)_{\tor})$ consists precisely of the
sums of squares. Hence the statement for $\mathbf{W^\times}$. 

The statement for $\mathbf{GW^\times}$ follows from this together with
\prettyref{cor:gww}, since $-1\in
GW(F)$ is unramified, hence does not contribute to the quotient.
\end{proof}

\begin{remark}
The above implies in particular that we have computed the first 
contraction $\mathbf{W^\times}_{-1}$. In that language, the above
result reads
$$
\mathbf{W^\times}_{-1}(L)\cong\mathcal{A}\oplus W(L)_{\tor}^{(1)}.
$$
\end{remark}

\begin{proposition}
Let $F$ be a field in $\mathcal{F}_k$. Then we have the following
short exact sequence 
$$
1\rightarrow GW(F)^\times\rightarrow
GW(F(T))^\times\stackrel{\sum \partial_{(\pi)}^\pi}{\longrightarrow} 
\bigoplus_{\pi\in F[T]}\left(\mathcal{A}_\pi\oplus
  W(F[T]/(\pi))_{\tor}^{(1)}\right)\rightarrow 0, 
$$
where the direct sum is taken over all irreducible monic polynomials.
\end{proposition}

\begin{proof}
As in the proof of \prettyref{prop:contract}, it suffices to consider
$W(F)^\times$ instead of $GW(F)^\times$.
Recall Milnor's split exact sequence \cite[Theorem 5.3]{milnor}
$$
0\rightarrow W(F)\rightarrow W(F(T))\rightarrow
\bigoplus_{\pi} W(F[T]/(\pi))\rightarrow 0. 
$$
This immediately settles injectivity $W(F)^\times\hookrightarrow
W(F(T))^\times$. 

For surjectivity of $\sum \partial_{(\pi)}^\pi$, let
$\overline{\alpha}\in\bigoplus_{\pi\in F[T]}W(F[T]/(\pi)))$ be
torsion. Using the splitting of Milnor's sequence, there exists
$\alpha\in W(F(T))_{\tor}$ lifting $\overline{\alpha}$. Setting
$1+\alpha\in W(F(T))^\times$, we have
$$
\left(\sum\partial_{(\pi)}^\pi\right)(1+\alpha)=\overline{\alpha}.
$$
If $\pi$ is such that $\mathcal{A}_\pi$ is $\mathbb{Z}/2$, then
$\langle\pi\rangle$ provides a lift of $1\in \mathcal{A}_\pi$. 

Finally, exactness in the middle. Let $\langle u\rangle\in
F(T)^\times/(F(T)^\times)^2$ be a square class unit with
$(\sum\partial_{(\pi)}^\pi)(\langle u\rangle)=0$. Decomposing $u$ into
primes, it suffices to consider the two cases $u=\pi$ irreducible
monic and $u\in F^\times$ constant. In the second case, $u$ is
obviously in the image and we are done. 
Let $1+a\in (1+I(F(T))_{\tor})$ with
$(\sum\partial_{(\pi)}^\pi)(1+a)=0$. By \prettyref{prop:contract},
this means that $a$ maps to $0$ under the residue morphism
$W(F(T))\rightarrow \bigoplus_{\pi}W(F[T]/(\pi))$ in Milnor's
sequence. Exactness of Milnor's sequence implies that $a$ is in the
image of $W(F)_{\tor}$. Hence the unit $1+a$ lies in the image of
$W(F)^\times$. 
\end{proof}

We record a similar but much easier computation for the torsion in the
Witt groups. This will lead to an identification of the higher
contractions of the units.

\begin{proposition}
Let $F$ be a field in $\mathcal{F}_k$, let $v$ be a discrete
valuation on $F$ and let $\pi$ be a uniformizing element for $v$. Then
there exists an isomorphism 
$$
H^1_v(\mathcal{O}_v,\mathbf{W_{\tor}^{(n)}})=
W(F)_{\tor}^{(n)}/\mathbf{W_{\tor}^{(n)}}(\mathcal{O}_v)
\stackrel{\cong}{\longrightarrow} 
W(\kappa(v))_{\tor}^{(n+1)}.
$$
Moreover, there is a short exact sequence 
$$
1\rightarrow W(F)_{\tor}^{(n)}\rightarrow
W(F(T))_{\tor}^{(n)}\stackrel{\sum \partial_{(\pi)}^\pi}{\longrightarrow} 
\bigoplus_{\pi\in F[T]}  W(F[T]/(\pi))_{\tor}^{(n+1)}\rightarrow 0, 
$$
where the direct sum is taken over all irreducible monic polynomials.
\end{proposition}

\begin{proof}
The second part without the modifying $(n)$-s is a direct consequence of
Milnor's split exact sequence \cite[Theorem 5.3]{milnor}, restricted
to the torsion. 

For the first part, note that the morphism $W(F)_{\tor}\rightarrow
W(\kappa(v))_{\tor}$ is simply the residue restricted to
torsion, and it is split by $a\mapsto (\langle\pi\rangle-1)a$. As a
set map, it is the same residue morphism as for the Witt groups
with the usual addition. Therefore, as a set map, injectivity of
$\theta_\pi$ is the Gersten conjecture for Witt groups in a very
simple case, surjectivity follows as in \prettyref{lem:contract}. 

We need to check the homomorphism property. By the form of the
splitting mentioned above, we need to compute the product
\begin{eqnarray*}
\left((\langle\pi\rangle-1)a\right)\boxplus_n
\left((\langle\pi\rangle-1)b\right)&=&
(\langle\pi\rangle-1)(a+b)+(-2)^n(\langle\pi\rangle-1)^2ab \\
&=&(\langle\pi\rangle-1)(a+b+(-2)^{n+1}ab).
\end{eqnarray*}
This now implies both claims.
\end{proof}

\begin{remark}
The above can be reformulated in the language of contractions as
follows: for $n\geq 2$, we have an isomorphism of abelian groups
$$
\mathbf{GW}^\times_{-n}(L)\cong W(L)_{\tor}^{(n)}.
$$
\end{remark}

\section{Gersten resolution and strict 
\texorpdfstring{$\Ao$}{A1}-invariance}
\label{sec:gersten}

In this section, we discuss a Gersten-type resolution for the
unramified sheaf $\mathbf{GW^\times}$. The description of units in
\prettyref{prop:units} provides an exact sequence which decomposes
$\mathbf{GW^\times}$ into a part coming from $\mathbf{K_1/2}$ and a part
$\mathbf{NQ}$ which is closely related to the torsion in the
fundamental ideal $\mathbf{I_{\tor}}$. From the Gersten resolution, we
can deduce strict $\Ao$-invariance. 

\subsection{The unramified sheaf 
\texorpdfstring{$\mathbf{NQ}$}{NQ}} 

We first introduce an unramified sheaf $NQ$. Recall from
\prettyref{prop:units} that the units in the Grothendieck-Witt ring
decompose into a square-class part and a part coming from nilpotent
elements in the Witt ring. The sheaf $NQ$ deals with the part of the
units of the Witt ring which comes from nilpotent elements:

\begin{definition}
\label{def:nq}
For a field $F$, we denote 
$$
NQ(F)=(1+I(F)_{\tor})/\left(F^\times/(F^\times)^2\cap
  (1+I(F)_{\tor})\right)\cong W^\times(F)/(F^\times),
$$
which is the quotient of the horizontal maps of the pushout square in 
\prettyref{prop:units}. 
\end{definition}

The notation $NQ$ is
supposed to suggest that this group is the \textbf{Q}uotient of the
subgroup of the units coming from \textbf{N}ilpotent elements of the
Witt ring. 

Since $NQ$ is defined as the quotient of $W^\times$
modulo the square classes $\mathbb{G}_m/2$, the data (D1)-(D3) for
$W^\times$ induce corresponding data for $NQ$. With these data, $NQ$
extends to an unramified sheaf of abelian groups. For an irreducible
smooth scheme $X$ with function field $F$, we denote by  
$$
\mathbf{NQ}(X)=\bigcap_{x\in
  X^{(1)}}\mathbf{NQ}(\mathcal{O}_x)\subseteq \mathbf{NQ}(F)
$$ 
the subgroup of unramified elements in $NQ(F)$. 

\subsection{Gersten resolution for
  \texorpdfstring{$I_{\tor}$}{torsion}}

We  next discuss the Gersten resolution for the torsion in the
fundamental ideal of the Witt ring. 
The assignment $L\in\mathcal{F}_k\mapsto I^\ast(L)$ mapping a field to
the $\mathbb{Z}$-graded abelian group of powers $I^n(L)$ of the
fundamental ideal in the Witt ring $W(L)$ extends to a
$\mathbb{Z}$-graded family of strictly $\Ao$-invariant sheaves of
groups $\mathbf{I^\ast}$ on $\smk$. This is a 
consequence of Example 3.34, Lemma 3.35 and Theorem 2.46 of
\cite{morel:book}. 

Now we consider for $n\in \mathbb{N}$ the morphism
$2^n:\mathbf{I^\ast}\rightarrow \mathbf{I^\ast}$. As this concerns
only the abelian group structure of $I^\ast$, it commutes with all the
data (D1)-(D4) making $\mathbf{I^\ast}$ an unramified sheaf of abelian
groups.  Using Lemma 3.35 and Theorem 2.46 of \cite{morel:book} again,
we have the following result:

\begin{lemma}
The assignment 
$$
I^\ast[2^n]:\mathcal{F}_k\rightarrow \mathcal{A}b:L\mapsto
\ker\left(2^n:I^\ast\rightarrow I^\ast\right)
$$
extends to a strictly $\Ao$-invariant sheaf of abelian groups. 
The same is true for the filtered colimit
$$
I^\ast_{\tor}=I^\ast[2^\infty]=\operatorname{colim}I^\ast[2^n].
$$ 
\end{lemma}

Using this together with the results from \cite{morel:book}, we have a
Gersten resolution for $\mathbf{I_{\tor}}$.  

\begin{lemma}
\label{lem:itor}
\begin{enumerate}[(i)]
\item 
There is a complex
$$
0\rightarrow \mathbf{I_{\tor}}(X)\rightarrow
I(K)_{\tor}\rightarrow\bigoplus_{x_1\in
  X^{(1)}} W(\kappa(x_1))_{\tor}\rightarrow \cdots \rightarrow
\bigoplus_{x_{n}\in
  X^{(n)}}W(\kappa(x_{n}))_{\tor}\rightarrow 0
$$
for any essentially smooth $k$-scheme $X$ of dimension $n$ with
function field $K$.  
\item This complex is a Nisnevich flasque resolution of
  $\mathbf{I_{\tor}}$. 
\item For $X$ essentially smooth and local, the complex is exact. 
\item The cohomology of this complex is $\Ao$-invariant for $X$ an
  essentially smooth scheme. 
\end{enumerate}
\end{lemma}

\begin{proof}
By \cite[Corollary 5.44]{morel:book}, the Gersten resolution can be
identified with what Morel calls Rost-Schmid complex.  

(i) follows from \cite[Theorem 5.31]{morel:book}, (ii) is 
\cite[Corollary 5.43]{morel:book},(iii) is \cite[Theorem
5.41]{morel:book}, and (iv) is \cite[Theorem
5.38]{morel:book}.
\end{proof}  

\subsection{Resolution for \texorpdfstring{$\mathbf{NQ}$}{NQ}}

The next step is an appropriate modification of the resolution for
$\mathbf{I_{\tor}}$ to obtain a resolution for $\mathbf{NQ}$. 

\begin{lemma}
\label{lem:nq}
\begin{enumerate}[(i)]
\item 
There is a complex
$$
1\rightarrow \mathbf{1+I_{\tor}}(X)\rightarrow
1+I(K)_{\tor}\rightarrow\bigoplus_{x_1\in
  X^{(1)}} W(\kappa(x_1))_{\tor}^{(1)}\rightarrow $$
$$\rightarrow\cdots \rightarrow
\bigoplus_{x_{n}\in
  X^{(n)}}W(\kappa(x_{n}))_{\tor}^{(n)}\rightarrow 0
$$
for any essentially smooth $k$-scheme $X$ of dimension $n$ with
function field $K$. The cohomology of this complex is $\Ao$-invariant
for $X$ an essentially smooth scheme, and the complex is exact for $X$
an essentially smooth local scheme. 
\item The complex of (i) contains the following subcomplex: 
$$
1\rightarrow \mathbb{G}_m/2(X)\cap \mathbf{1+I_{\tor}}(X)\rightarrow 
(K^\times/(K^\times)^2)\cap (1+I(K)_{\tor})\rightarrow
\bigoplus_{x_1\in X^{(1)}}\mathcal{S}_{\pi(x_1)}\rightarrow 0,
$$
where $\pi(x_1)$ is a uniformizer for the point $x_1$ and
$\mathcal{S}_{\pi(x_1)}=\mathbb{Z}/2$ or $0$ if the square class of
the uniformizer $\pi$ is a sum of squares or not, respectively.
The cohomology of this complex is $\Ao$-invariant
for $X$ an essentially smooth scheme, and the complex is exact for $X$
an essentially smooth local scheme. 
\item The statements in (i) also hold for 
$$
1\rightarrow \mathbf{NQ}(X)\rightarrow
NQ(K)\rightarrow\bigoplus_{x_1\in
  X^{(1)}} W(\kappa(x_1))_{\tor}^{(1)}/\mathcal{S}_{\pi(x_1)}\rightarrow $$
$$\rightarrow\cdots \rightarrow
\bigoplus_{x_{n}\in
  X^{(n)}}W(\kappa(x_{n}))_{\tor}^{(n)}\rightarrow 0
$$

\end{enumerate}
\end{lemma}

\begin{proof}
(i) is a direct consequence of \prettyref{lem:itor}. The complex is
obtained by taking the complex from \prettyref{lem:itor}, and changing
the group structures: $1+I(F)_{\tor}$ as a set is the same as
$I(F)_{\tor}$, but with multiplication instead of addition. Similarly,
the underlying sets of $W(F)_{\tor}^{(n)}$ and $W(F)_{\tor}$ are the
same, but the addition is different. From \prettyref{sec:contract}, we
know that the differentials of the complex of \prettyref{lem:itor} are
homomorphisms for the modified group structures. Moreover, the neutral
elements for $W(F)_{\tor}^{(n)}$ and $W(F)_{\tor}$ are the
same. Therefore, the property $\partial^2=0$ claimed in the present
lemma does not depend on the abelian group structure, only on the
underlying set maps and the neutral elements - hence it follows from
\prettyref{lem:itor}. Similarly, vanishing of cohomology in the case
of smooth local schemes does not depend on the abelian group
structure, only on the underlying sets and set maps. Hence it follows
from \prettyref{lem:itor}.

(ii) From the Gersten resolution for $K^M_1/2$, we know that the claim
holds for 
$$
1\rightarrow \mathbb{G}_m/2(X)\rightarrow
K^\times/(K^\times)^2\rightarrow 
\bigoplus_{x_1\in X^{(1)}}\mathbb{Z}/2\rightarrow 0.
$$
Now any $u\in K^\times$, is either unramified or generates the
corresponding $\mathbb{Z}/2$-copy in the quotient. Therefore,
restricting to the square classes in $1+I(K)_{\tor}$ provides the
diagram in (ii), hence proves the claim. In particular, the complex is
exact for $X$ an essentially  smooth local scheme.

(iii) follows by taking the quotient of the complex in (i) by the
subcomplex in (ii). Exactness for essentially smooth local schemes
follows by the long exact sequence for the corresponding exact
sequence of complexes. The same is true for $\Ao$-invariance.
\end{proof}

\subsection{Resolution for 
\texorpdfstring{$\mathbf{GW^\times}$}{units}}

\begin{proposition}
\label{prop:exact}
Let $X$ be an essentially smooth scheme of dimension $n$ over $X$ with
function field $K$.
We have the following exact sequence of complexes:
\begin{center}
\begin{minipage}[c]{12cm}
\begin{tikzpicture}[scale=1.1,arrows=->]
\node (L0b) at (2,10) {$0$};
\node (L0c) at (5,10) {$0$};
\node (L0d) at (8.5,10) {$0$};

\node (L1a) at (0.5,9) {$0$};
\node (L1b) at (2,9) {$\mathbb{G}_m/2(X)$};
\node (L1c) at (5,9) {$\mathbf{GW^\times}(X)$};
\node (L1d) at (8.5,9) {$\mathbf{NQ}(X)$};
\node (L1e) at (10.5,9) {$0$};
\draw (L1a) to  (L1b);
\draw (L1b) to  (L1c);
\draw (L1c) to  (L1d);
\draw (L1d) to  (L1e);
\draw (L0b) to (L1b);
\draw (L0c) to (L1c);
\draw (L0d) to (L1d);

\node (L2a) at (0.5,8) {$0$};
\node (L2b) at (2,8) {$K^\times/(K^\times)^2$};
\node (L2c) at (5,8) {$GW(K)^\times$};
\node (L2d) at (8.5,8) {$NQ(K)$};
\node (L2e) at (10.5,8) {$0$};
\draw (L2a) to  (L2b);
\draw (L2b) to  (L2c);
\draw (L2c) to  (L2d);
\draw (L2d) to  (L2e);
\draw (L1b) to (L2b);
\draw (L1c) to (L2c);
\draw (L1d) to (L2d);

\node (L3a) at (0.5,7) {$0$};
\node (L3b) at (2,7) {$\bigoplus\mathbb{Z}/2$};
\node (L3c) at (5,7) {$\bigoplus\left(\mathcal{A}_{x_1}\oplus W(\kappa(x_1))_{\tor}^{(1)}\right)$};
\node (L3d) at (8.5,7) {$\bigoplus W(\kappa(x_1))_{\tor}^{(1)}/\mathcal{S}_{x_1}$};
\node (L3e) at (10.5,7) {$0$};
\draw (L3a) to  (L3b);
\draw (L3b) to  (L3c);
\draw (L3c) to  (L3d);
\draw (L3d) to  (L3e);
\draw (L2b) to (L3b);
\draw (L2c) to (L3c);
\draw (L2d) to (L3d);

\node (L4b) at (2,6) {$0$};
\node (L4c) at (5,6) {$\bigoplus W(\kappa(x_2))_{\tor}^{(2)}$};
\node (L4d) at (8.5,6) {$\bigoplus W(\kappa(x_2))_{\tor}^{(2)}$};
\node (L4e) at (10.5,6) {$0$};
\draw (L4b) to  (L4c);
\draw (L4c) to  (L4d);
\draw (L4d) to  (L4e);
\draw (L3b) to (L4b);
\draw (L3c) to (L4c);
\draw (L3d) to (L4d);

\node (L5c) at (5,5) {$\vdots$};
\node (L5d) at (8.5,5) {$\vdots$};
\draw (L4c) to (L5c);
\draw (L4d) to (L5d);

\node (L6b) at (2,4) {$0$};
\node (L6c) at (5,4) {$\bigoplus W(\kappa(x_n))_{\tor}^{(n)}$};
\node (L6d) at (8.5,4) {$\bigoplus W(\kappa(x_n))_{\tor}^{(n)}$};
\node (L6e) at (10.5,4) {$0$};
\draw (L6b) to  (L6c);
\draw (L6c) to  (L6d);
\draw (L6d) to  (L6e);
\draw (L5c) to (L6c);
\draw (L5d) to (L6d);

\node (L7c) at (5,3) {$0$};
\node (L7d) at (8.5,3) {$0$};
\draw (L6c) to (L7c);
\draw (L6d) to (L7d);

\end{tikzpicture}
\end{minipage}
\end{center}
The vertical morphisms are the boundary maps discussed in
\prettyref{sec:contract}. The omitted index sets of the direct sums
are the respective sets of codimension $i$ points $x_i\in X^{(i)}$. 
\end{proposition}

\begin{proof}
(i) The diagram is commutative. This is almost by definition: the
second line is the one coming from the definition of $NQ$ resp. the
pushout square of \prettyref{prop:units}. The first line is the
restriction to unramified elements, so these squares commute. The
further lines are those computed in \prettyref{sec:contract}, so these
squares also do commute.

(ii) The left column is a complex. In fact, it is the Gersten
resolution for the strictly $\Ao$-invariant sheaf $\mathbf{K_1^M/2}$. 

(iii) The right column is a complex by \prettyref{lem:nq}. 

(iv) The rows of the diagram are exact. This follows from the
description of units, cf. \prettyref{prop:units}. In fact, $NQ$ is
defined in such a way that the second row is exact, whence also the
exactness of the first row. In the third row, we have the following
case distinction: if $\pi$ is a sum of squares, then $\mathcal{A}=0$
and $\mathcal{S}\cong\mathbb{Z}/2$, in which case we have an exact
sequence dividing out a copy of $\mathbb{Z}/2$ from
$W(\kappa(x_1))_{\tor}^{(1)}$. If $\pi$ is not a sum of squares, then
$\mathcal{A}\cong\mathbb{Z}/2$, and $\mathcal{S}=0$, in which case we
have an exact sequence dividing out the additional $\mathbb{Z}/2$ in
the middle. All rows below the third are obviously exact. 

(v) It follows by a diagram chase that the middle column is also a
complex. 
\end{proof}

\subsection{Strict \texorpdfstring{$\Ao$}{A1}-invariance}

Recall that a sheaf of groups $\mathcal{G}$ on $\smk$ is called
\emph{strictly $\Ao$-invariant} if for each $i$ and each smooth scheme
$X$ over $k$ the associated morphism
$H^i_{\Nis}(X,\mathcal{G})\rightarrow
H^i_{\Nis}(X\times\Ao,\mathcal{G})$ is an isomorphism. 

\begin{theorem}
Let $X$ be an essentially smooth local scheme of dimension $n$ over
$k$, let $K$ denote its function field, and let $z$ denote its closed
point. Then the following sequence is exact:
$$
0\rightarrow \mathbf{GW^\times}(X)\rightarrow
GW(K)^\times\rightarrow\bigoplus_{x_1\in
  X^{(1)}}\left(\mathcal{A}_x\oplus
  W(\kappa(x_1))_{\tor}^{(1)}\right)\rightarrow 
$$
$$
\rightarrow \bigoplus_{x_2\in
  X^{(2)}}W(\kappa(x_2))_{\tor}^{(2)}\rightarrow\cdots \rightarrow
\bigoplus_{x_{n-1}\in
  X^{(n-1)}}W(\kappa(x_{n-1}))_{\tor}^{(n-1)}\rightarrow 
W(\kappa(z))_{\tor}^{(n)}\rightarrow 0. 
$$
\end{theorem}

\begin{proof}
Using \prettyref{prop:exact}, we immediately reduce to the exactness
of the resolution of $\mathbf{NQ}$ which follows from \prettyref{lem:nq}. 
\end{proof}

\begin{theorem}
\label{thm:sa1}
The unramified sheaf $\mathbf{GW^\times}$ is strictly $\Ao$-invariant.
\end{theorem}

\begin{proof}
Again, we use \prettyref{prop:exact} and the Gersten resolution for
$K^M_1/2$ to reduce to the $\Ao$-invariance of $\mathbf{NQ}$ which
follows \prettyref{lem:nq}.
\end{proof}

From the description of the units in the
Grothendieck-Witt ring, we obtain a long exact sequence decomposing
the $\mathbf{GW^\times}$-homology into a weight one $K$-theory part and
a modified Witt-ring homology part. This is a direct consequence of
\prettyref{prop:exact}. 

\begin{proposition}
\label{prop:nq}
There is a short exact sequence of strictly $\Ao$-invariant sheaves of
groups on $\smk$
$$
1\rightarrow\mathbb{G}_m/2\rightarrow \mathbf{GW^\times}\rightarrow
\mathbf{NQ}\rightarrow 1.
$$
This induces (functorially on $\smk$) an  exact sequence 
$$
0\rightarrow
\mathbb{G}_m/2(X)\rightarrow 
\mathbf{GW^\times}(X)\rightarrow \mathbf{NQ}(X)\rightarrow 
\operatorname{Pic}/2(X)\rightarrow 
$$
$$
\rightarrow H^1_{\Nis}(X,\mathbf{GW}^\times)
\rightarrow H^1_{\Nis}(X,\mathbf{NQ})\rightarrow 0
$$
as well as (for $i\geq 2$) isomorphisms
$$
H^i_{\Nis}(X,\mathbf{GW}^\times)\stackrel{\cong}{\longrightarrow}
H^i_{\Nis}(X,\mathbf{NQ}).
$$
\end{proposition}

Note that from the proof above, we have a bijection of sets
$H^i(X,\mathbf{NQ})\cong H^i(X,\mathbf{W_{\tor}})$ for $i\geq
2$. However, as the abelian group structures of $W_{\tor}$ and
$W_{\tor}^{(n)}$ are different, the above bijection of sets does not
respect the group structure. Nevertheless, non-triviality of elements
in $H^i(X,\mathbf{NQ})$ can be detected in $H^i(X,\mathbf{W_{\tor}})$.

\section{\texorpdfstring{$\Ao$}{A1}-spherical fibrations and
  orientation theory}  
\label{sec:sphere}

In this final section, we discuss consequences of the previous results
for orientation theory and spherical fibrations in $\Ao$-homotopy
theory. 

\subsection{Classifying space of spherical fibrations}

\begin{proposition}
The space $\classify{S^{2n,n}}$ is $\Ao$-local. Hence it is in fact
the classifying space of (Nisnevich locally trivial) spherical
fibrations. 
\end{proposition}

\begin{proof}
This is a direct consequence of \prettyref{thm:sa1} and
\cite[Theorem 8.1]{flocal} resp. \cite[Corollary 8.2]{flocal}.
\end{proof}

As a consequence, we can consider an unstable version of the
$J$-homomorphism on the classifying space level. 

\begin{proposition}
The homomorphism $GL_n\rightarrow \haut{S^{2n,n}}$ induces a morphism
of $\Ao$-local simplicial sheaves
$$
J^{\Ao}(n):L_{\Ao}BGL_n\rightarrow \classify{S^{2n,n}}.
$$
\end{proposition}

\begin{proof}
We describe the homomorphism: a matrix $M\in GL_n$ induces a scheme
morphism $M:\mathbb{A}^n\rightarrow\mathbb{A}^n$. Since this morphism
preserves $\mathbb{A}^n\setminus\{0\}$, it descends to a morphism 
$$
M:\mathbb{A}^n/(\mathbb{A}^n\setminus\{0\})\rightarrow
\mathbb{A}^n/(\mathbb{A}^n\setminus\{0\}).  
$$
This construction evidently maps matrix multiplication to
composition. We choose a functorial fibrant replacement
$\mathbb{A}^n/(\mathbb{A}^n\setminus\{0\})\rightarrow S^{2n,n}$  (note
the abuse of notation), and have thus described the homomorphism
$GL_n\rightarrow\haut{S^{2n,n}}$. A homomorphism between simplicial
monoids induces a morphism of the corresponding simplicial classifying
spaces, and we get the required morphism between the $\Ao$-local
simplicial sheaves by a functorial fibrant replacement again.
\end{proof}

The morphisms $BGL_n\rightarrow \classify{S^{2n,n}}$ stabilize
to a morphism $BGL_\infty\rightarrow \classify{S^{2\infty,\infty}}$, where on
the source we stabilize by adding a trivial line bundle, and on the
target we stabilize by fibrewise suspension with $S^{2,1}$. 
We denote by $G/O(n)$ the space 
$$
G/O(n)=\hofib\left(J^{\Ao}(n):L_{\Ao}BGL_n\rightarrow\classify{S^{2n,n}}\right), 
n\in \mathbb{N}\cup\{\infty\}.
$$
This space controls the difference between the classification of
rank $n$ vector bundles and the rank $n$ classification of spherical
fibrations. The notation is chosen to fit the usual topological
notation in which $G/O$ is the homotopy fibre of the $J$-homomorphism
$BO\rightarrow  BG$.

It would be very interesting to study the morphism
$[X,J^{\Ao}(n)]:[X,BGL_n]_{\Ao}\rightarrow[X,\classify{S^{2n,n}}]_{\Ao}$,
i.e. really compare vector bundles and spherical fibrations. However,
lack of descriptions of higher homotopy groups of
$\classify{S^{2n,n}}$ restricts us to the study of the first Postnikov
section of $J^{\Ao}(n)$, i.e. the homomorphism
$$
[X,J^{\Ao}(n)^{(1)}]:[X,B\mathbb{G}_m]_{\Ao}\rightarrow
[X,B\pi_1^{\Ao}\classify{S^{2n,n}}]_{\Ao}\cong 
[X,B\mathbf{GW^\times}]_{\Ao}.
$$
This is what we will do in the rest of the section. 

\subsection{Cohomology with unit coefficients}

In this section, we discuss the cohomology of
$\mathbf{GW^\times}$. 

\begin{proposition}
We have 
$$
\pi_1^{\Ao}J^{\Ao}(n):\mathbb{G}_m\rightarrow
\mathbf{GW^\times}: u\in F^\times\mapsto \langle u\rangle \in
GW(F)^\times.
$$
The composition $\operatorname{Pic}(X)\rightarrow
\operatorname{Pic}/2(X)\rightarrow H^1_{\Nis}(X,\mathbf{GW}^\times)$
of the obvious projection with the morphism 
from \prettyref{prop:nq} is precisely the induced morphism 
$$
H^1_{\Nis}(X,J^{\Ao}(n)^{(1)}):H^1_{\Nis}(X,\mathbb{G}_m)\rightarrow
H^1_{\Nis}(X,\mathbf{GW}^\times). 
$$
\end{proposition}

\begin{proof}
From the description of $J$, to determine the morphism induced on
fundamental groups, we need to consider for $u\in k^\times$ the
homotopy class of multiplication with the matrix
$\operatorname{diag}(u,1,\dots,1)$ in
$[\mathbb{A}^n\setminus\{0\},\mathbb{A}^n\setminus\{0\}]_{\Ao}$. Under
the isomorphism
$[\mathbb{A}^n\setminus\{0\},\mathbb{A}^n\setminus\{0\}]_{\Ao}\cong
GW(k)$, this is mapped to $\langle u\rangle$, cf. \cite{morel:book}. 
\end{proof}

\begin{remark}
\label{rem:easy}
Note that the above implies that the image of $\mathbb{G}_m\rightarrow
\mathbf{GW^\times}$ is then determined by the image of
$\mathbb{G}_m\rightarrow\mathbf{W^\times}$ since it factors through
the $W^\times$-coset of $1$ in the extension 
$$
1\rightarrow\{\pm 1\}\rightarrow
\mathbf{GW^\times}\rightarrow\mathbf{W^\times}\rightarrow 1.
$$
\end{remark}

\begin{proposition}
\label{prop:spheres}
We have the following computation of Nisnevich cohomology of
$\mathbf{GW^\times}$ over spheres:
$$
H^i_{\Nis}(S^p\wedge\mathbb{G}_m^{\wedge q},\mathbf{GW^\times})\cong 
H^{i-p}_{\Nis}(\Spec k,(\mathbf{GW^\times})_{-q})\cong 
\left\{\begin{array}{ll}
0&i\neq p\\
GW(k)^\times & q=0,\\
\mathcal{A}\oplus W(k)_{\tor}^{(1)} & q= 1,\\
W(k)_{\tor}^{(q)} & q\geq 2.
\end{array}\right.
$$
In the above, $\mathcal{A}=0$ if $T\in k(T)$ is a sum of squares and
$\mathcal{A}\cong\mathbb{Z}/2$ otherwise. 
\end{proposition}

\begin{proof}
The first isomorphism is standard, it could be called ``Thom
isomorphism''. The second restates the computations from
\prettyref{sec:contract}. 
\end{proof}

\begin{remark}
The tangent bundle of $\mathbb{P}^1$ is $\mathcal{O}(-2)$, hence its
class in the cohomology group
$H^1_{\Nis}(\mathbb{P}^1,\mathbf{GW^\times})$ is trivial. The
$\mathbb{Z}/2$-summand in the cohomology group is generated by
$\mathcal{O}(-1)$, the M\"obius strip.  

There can be many additional $\Ao$-homotopy classes of morphisms
$\mathbb{P}^1\rightarrow B\mathbf{GW^\times}$. For example, if
$k=\mathbb{Q}$, the Witt ring $W(k)$ contains a lot of torsion
elements. However, none of these homotopy classes is visible in either
real or complex realization: for the complex realization 
$W(\mathbb{C})=\mathbb{Z}/2$, for the real realization
$W(\mathbb{R})=\mathbb{Z}$ is torsion-free.
\end{remark}

\subsection{Orientation theory}

\begin{definition}
A spherical fibration $f\in [X,\classify{S^{2n,n}}]_{\Ao}$ is called
\emph{orientable} if it factors through  the universal $\Ao$-covering
$$
\widetilde{\classify{S^{2n,n}}}=B(\haut{S^{2n,n}})_0\rightarrow
\classify{S^{2n,n}}.
$$
Given a spherical fibration $f:X\rightarrow \classify{S^{2n,n}}$ with
$X$ an $\Ao$-connected space, we define its \emph{orientation
  character} to be the morphism  
$$
\pi_1^{\Ao}(f):\pi_1^{\Ao}(X)\rightarrow
\pi_1^{\Ao}\classify{S^{2n,n}}\cong \mathbf{GW^\times}.
$$
\end{definition}

\begin{remark}
\begin{itemize}
\item
Obviously, a spherical fibration over an $\Ao$-simply-connected base
$X$ is orientable.  
\item There is a similar notion of orientability for vector bundles
  requiring that the classifying map $X\rightarrow BGL_n$ factors
  through the inclusion $BSL_n\rightarrow BGL_n$. This is equivalent
  to the vector bundle having trivial first Chern class.
\item More generally, we could define an orientation character to be
  the induced morphism on fundamental groupoids in case the space $X$
  is not $\Ao$-connected. As we only deal with $\mathbb{P}^1$ in what 
  follows, we keep it simple.
\end{itemize}
\end{remark}

The above notion of orientability is the one employed in
topology for spherical fibrations. In topology, orientability of
vector bundles and their associated spherical fibrations are
equivalent since $\pi_1(O(n))\cong\pi_1(G)\cong\mathbb{Z}/2$. However,
this is not the case in $\Ao$-homotopy theory. The following
comparison of the two notions of orientability is a direct consequence
of \prettyref{prop:nq}:

\begin{proposition}
Let $f\in [X, BGL_n]_{\Ao}$ be a (continuous) vector bundle on
$X$. The associated spherical fibration is orientable if one of the
following conditions holds:
\begin{enumerate}[(i)]
\item  $c_1(E)\equiv 0\mod 2$.
\item The class $[f]\in \operatorname{Pic}/2$ is in the image of $\mathbf{NQ}$ of
  \prettyref{prop:nq}. 
\end{enumerate}
\end{proposition}

We use the proposition to discuss some examples of the arithmetic
nature of the above notion of orientability. 

\begin{example}
The orientation character of the tangent bundle of
  $\mathbb{P}^n$ is given by $\mathbb{G}_m\rightarrow
  \mathbf{GW}^\times:T\mapsto \langle T^{n+1}\rangle$. We denote by
  $ST(n)$ the spherical fibration associated to the tangent
  bundle of $\mathbb{P}^n$. We use that we can determine the morphism
  $\mathbb{G}_m\rightarrow\mathbf{GW^\times}$ from its image in
  $\mathbf{W^\times}$, cf. \prettyref{rem:easy}.
\begin{enumerate}[(i)]
\item Over an algebraically closed base $k=\overline{k}$,
  $ST(n)$ is orientable for any $n$: in that case $W(k)=\mathbb{Z}/2$,
  so $W(k)^\times=\{1\}$. 
\item Over the real numbers $k=\mathbb{R}$, the morphism
  $\pi_1^{\Ao}BGL_n\rightarrow \mathbf{W^\times}$ looks as follows:
  $$
  k^\times\rightarrow W(\mathbb{R})^\times:u\mapsto 
  \left\{\begin{array}{ll} 
      1 & u\in(\mathbb{R}^\times)^2\\
      -1 & u\not\in(\mathbb{R}^\times)^2
    \end{array}\right.
  $$
Composing this with the orientation character we find that $ST(n)$ is
orientable if and only if $n$ is odd. This is 
  the behaviour of orientability for real projective spaces. 
\item In characteristic $p$, every element is a sum of
  squares. Therefore, $\langle T\rangle$ lies in the image of
  $\mathbf{NQ}$. Hence, in characteristic $p$, $ST(n)$ is always
  orientable. 
\end{enumerate}
\end{example}

\subsection{Reducing spherical fibrations to vector bundles}
We comment on the first obstruction for spherical fibrations to be the
induced spherical fibrations of vector bundles. 

\begin{proposition}
We have an isomorphism of strongly $\Ao$-invariant sheaves of groups 
$$
\pi_0^{\Ao}(G/O)\cong \mathbf{NQ}.
$$
In particular, the space $G/O$ is \'etale $\Ao$-connected but not
necessarily $\Ao$-connected. 

Let $x\in G/O(\Spec k)$ be in the component of the
trivial spherical fibration. Then there is a surjection
$$
\pi_1^{\Ao}(G/O,x)\rightarrow 2\mathbb{G}_m\rightarrow 0.
$$
\end{proposition}

\begin{proof}
The part concerning $\pi_0^{\Ao}$ and $\pi_1^{\Ao}$ of the long exact
sequence associated to the fibre sequence 
$$
G/O\rightarrow BGL_n\rightarrow \classify{S^{2n,n}}
$$
is exactly the exact sequence of \prettyref{prop:nq}. This implies
the result on $\pi_0$,  noting that both $BGL_n$ and
$\classify{S^{2n,n}}$ are $\Ao$-connected. 

The result on $\pi_1$ also follows from the long exact sequence,
noting that $\mathbb{G}_m\cong \pi_1^{\Ao}BGL_n\rightarrow
\mathbf{GW^\times}$ factors through the quotient $\mathbb{G}_m/2$. 
\end{proof}

\begin{proposition}
\label{prop:curve}
Let $C$ be a smooth curve.
Then we have 
$$
[C,\classify{S^{2n,n}}]_{\Ao}\cong[C,B\mathbf{GW^\times}]_{\Ao}\cong
H^1_{\Nis}(C,\mathbf{GW^\times}).
$$
\end{proposition}

\begin{proof}
The last bijection is basically the definition of
$B\mathbf{GW^\times}$. 

The first bijection is obstruction theory: the space 
$B\mathbf{GW^\times}$ is the first Postnikov section of
$\classify{S^{2n,n}}$. The obstruction and lifting classes for lifting
a morphism $f:C\rightarrow \classify{S^{2n,n}}^{(i)}$ to the next
stage of the Postnikov tower lie in $H^{i+1}_{\Nis}$
resp. $H^{i+2}_{\Nis}$ of $C$ with coefficients in
$\pi_{i+1}^{\Ao}\classify{S^{2n,n,}}$. But these cohomology groups
vanish for reasons of cohomological dimension. Therefore, any morphism
$f\in [C,B\mathbf{GW}^\times]_{\Ao}$ has a (up to $\Ao$-homotopy)
unique extension to $\classify{S^{2n,n,}}$. 
\end{proof}

\begin{corollary}
We have 
$$
[\mathbb{P}^1,\classify{S^{2n,n}}]_{\Ao}\cong \mathcal{S}\oplus
W(k)_{\tor}^{(1)}. 
$$
\end{corollary}

\begin{proof}
This follows directly from \prettyref{prop:curve} and
\prettyref{prop:spheres}. 
\end{proof}

Therefore, we can explicitly describe all spherical fibrations over
$\mathbb{P}^1$. They are given by glueing two trivial spherical
fibrations over $\mathbb{P}^1\setminus\{0\}$ and
$\mathbb{P}^1\setminus\{\infty\}$ along a transition morphism
$\mathbb{G}_m\rightarrow \haut{S^{2n,n}}$. 

In particular, there exist spherical fibrations over $\mathbb{P}^1$
that are not associated to vector bundles. By \prettyref{prop:nq},
these are classified exactly by
$$
H^1_{\Nis}(\mathbb{P}^1,\mathbf{NQ})\cong H^0_{\Nis}(\Spec
k,(\mathbf{NQ})_{-1})\cong W(k)_{\tor}^{(1)}/\mathcal{S}.
$$

\begin{example}
Cazanave has described the automorphisms of $\mathbb{P}^1$ in
\cite{cazanave}. As a consequence of the results of \cite{cazanave}, the
rational function 
$$
\frac{X^3-\left(\frac{a_3}{a_2}+\frac{a_2}{a_1}\right)X}
{a_1X^2-\frac{a_1a_3}{a_2}}
$$
is the endomorphism associated to the quadratic form $\langle
a_1\rangle\oplus\langle a_2\rangle\oplus \langle a_3\rangle$ in $GW(k)$. 

In particular, for $k=\mathbb{Q}_p$ and $u$ any square class, the rational
function 
$$
\frac{X^3-\left(T+u\right)X}
{X^2-T}
$$
has associated quadratic form $1\oplus\langle T\rangle\langle
u\rangle\oplus\langle u\rangle$. Therefore, the above rational function is
an endomorphism of $\mathbb{P}^1$ which is invertible up to
$\Ao$-homotopy, but not $\Ao$-homotopic to a fractional linear
transformation. Tracing through our identifications
\prettyref{prop:spheres} and \prettyref{lem:contract}, we find that
gluing two trivial spherical fibrations over
$\mathbb{P}^1\setminus\{0\}$ and $\mathbb{P}^1\setminus\{\infty\}$
along the $\mathbb{G}_m$-family of endomorphisms given by the rational
functions  
$$
\frac{X^3-\left(T+u\right)X}
{X^2-T}
$$
provides a spherical fibration over $\mathbb{P}^1$ whose associated
class in $H^1_{\Nis}(\mathbb{P}^1,\mathbf{GW^\times})$ is
non-trivial. Similar examples can be given using quadratic forms over
$\mathbb{Q}$. 
\end{example}

The spherical fibrations above provide examples of exotic Poincar{\'e}
duality structure on $\mathbb{P}^1$. In the case of the base field
$\mathbb{Q}$, these structures are indistinguishable from the standard
Poincar{\'e} duality on $\mathbb{P}^1$ in real and complex realizations,
in the category of motives and in the $\mathbb{Z}[1/2]$-localized
stable $\Ao$-homotopy. In particular, there are several spherical fibrations
inducing the standard Poincar{\'e} duality, but only one of them
is induced from a vector bundle. The notion of Poincar{\'e} 
duality space, should it ever take form in $\Ao$-homotopy, necessarily
has to include the $2$-torsion information from the Witt group to
prevent this sort of pathological behaviour.

\end{document}